\documentclass[10pt]{amsart}
\usepackage{amssymb, amscd, amsmath, amsthm, latexsym, enumerate}

\renewcommand{\geq}{\geqslant}
\renewcommand{\leq}{\leqslant}

\newcommand{\Z}{\mathbb Z}

\newtheorem{theorem}{Theorem}
\newtheorem{lemma}[theorem]{Lemma}

\theoremstyle{definition}

\newtheorem*{cor*}{Corollary}

\begin{document}

\title{$3$-Manifolds with more than one abelian embedding}

\author{Jonathan A. Hillman }
\address{School of Mathematics and Statistics\\
     University of Sydney, NSW 2006\\
      Australia }

\email{jonathanhillman47@gmail.com}

\begin{abstract}
We construct examples of 3-manifolds $M$ 
which have at least two inequivalent embeddings in $S^4$ 
such that in each case the complementary regions have abelian fundamental groups.
\end{abstract}


\maketitle

A TOP locally flat embedding of a closed connected 3-manifold $M$ in $S^4$ is  {\it abelian\/}
if each of the fundamental groups $\pi_X$ and $\pi_Y$ of the
two complementary regions $X$ and $Y$ is abelian.
If $M$ has such an embedding then either $\beta=\beta_1(M)=1,3,4$ or $6$ and 
$H_1(M)=H_1(M;\Z)\cong\Z^\beta$,
or $H_1(M)\cong{C_n^2}$ or $\Z^2\oplus{C_n^2}$, for some $n>0$  \cite[Theorem 8.1]{LF}.
In all cases the abelian groups $\pi_X$ and $\pi_Y$ 
have balanced presentations.
If $M$ is a homology 3-sphere then it has an essentially unique abelian embedding
(and the complementary regions are then contractible), 
while if $H_1(M)\cong\Z$ then it has at most one such embedding \cite[Theorem 8.9]{LF}.

We shall show that there are examples with more than one abelian embedding.
Our strategy is to find a link $L$ which has several distinct partitions into a pair of sublinks, 
each of which is trivial (or split $Ap1$,  as defined below),
and to consider the associated embeddings of the manifold $M=M(L)$ 
obtained by 0-framed surgery on $L$.
For appropriate choices of $L$ the embeddings are abelian,
and we can use the essential uniqueness of the JSJ decomposition of $M$ 
to show that the embeddings are distinct.
One example with $\beta=6$  has (at least) 5 abelian embeddings.

At the end we attach a short section outlining how surgery may be applied
when $\beta=0$ and the complementary regions have fundamental group $C_n$, for some $n>0$.

\section{the examples}

Embeddings $j$ and $\tilde{j}$ of a 3-manifold $M$ in $S^4$ are {\it equivalent\/} if there are 
self-homeomorphisms $\phi$ of $M$ and $\psi$ of $S^4$ such that $\psi{j}=\tilde{j}\phi$.
Let $j_X:M\to{X}$ and $j_Y:M\to{Y}$ be the inclusions of $M$ into each 
of the complementary regions for the embedding $j$ (and similarly for $j'$).
In particular, if the image of the complementary regions $X$ and $Y$ under $\psi$ are $X'$ 
and $Y'$ then $H_1(\phi)$ maps the kernel of $H_1(j_X)$ onto the kernel of $H_1(j_{X'})$
and the kernel of $H_1(j_Y)$ onto the kernel of $H_1(j_{Y'})$.
Thus in order to show that two embeddings of $M$ are not equivalent it shall suffice to show that
there is no such automorphism of $H_1(M)$.

Our examples shall all be variations on the Borromean rings link $Bo$.
All the proper sublinks of $Bo$ are trivial links, 
and the exterior $X(Bo)$ is hyperbolic \cite[Exercise 3.3.10]{Th}.
We shall say that a knot $K$ in $S^3$ is $Ap1$ if it has Alexander polynomial $\Delta_K=1$.
Every $Ap1$ knot bounds a TOP locally flat disc in $D^4$ with complement having
fundamental group $\Z$ \cite[Theorem 11.7B]{FQ}.
A link in $S^3$ is {\it split $Ap1$} if it is a split link and each component is an $Ap1$ knot.
Such links are slice links, and have a set of slice discs with complement 
having free fundamental group.
We shall also arrange that the nontrivial components have hyperbolic exterior,
as this may simplify the invocation of JSJ arguments.
There are infinitely many hyperbolic $Ap1$ knots  \cite{Ka04}.

The simplest nontrivial $Ap1$ knot  is the Kinoshita-Teresaka knot $K=11_{42n}$,
which is an 11 crossing knot with $\Delta_K=1$.
It is hyperbolic, 
and bounds a smooth disc $D_K$ in $D^4$ such that $\pi_1(D^4\setminus{D_K})\cong\mathbb{Z}$.  
See \cite[Figure 1.4]{AIL}. 
(If we could find 4 other such hyperbolic $Ap1$ knots we could avoid any appeal to TOP surgery.)

Let $L$ be the 3-component link obtained from $Bo$ by replacing the  third component  
$Bo_3$ by a nontrivial $Ap1$ knot $K$.
(See Figure 1, in which $K_o\subset{D^3}$ is the tangle obtained by deleting 
a small ball around a point on $K\subset{S^3}$.)
The link $L$ has two distinct partitions into a pair of sublinks, 
each of which is a split $Ap1$ link: 
$\mathcal{P}=\{\{Bo_1,Bo_2\}, K\}$ and $\mathcal{P}'=\{\{Bo_1,K\}, K_2\}$.

Let $M=M(L)$ and let $j$ and $j':M\to {S^4}$ be the embeddings determined
by these partitions, together with the obvious slice discs (as in \cite[Chapter 2]{LF}).
It is easy to see that in each case $\pi_X\cong\mathbb{Z}^2$ and $\pi_Y\cong\mathbb{Z}$,
and so $j$ and $j'$ are abelian embeddings.
(In each case $X\simeq{T^2}$ and $Y\simeq{S^1\vee2S^2}$ \cite[Theorem 8.17]{LF}.)

\setlength{\unitlength}{1mm}
\begin{picture}(100,45)(-25,-7)

\put(7,2){$a$}
\put(29,28){$p$}
\put(32.5,23.3){$q$}
\put(30,16){$r$}
\put(50,32){$y$}
\put(52,2){$x$}
\put(63,16.5){$K_o$}

\put(7,4.1){$\to$}
\put(52,4.1){$\to$}
\put(32,21.1){$\to$}
\put(28,25.15){$\gets$}
\put(28,14.1){$\gets$}
\put(52, 29.1){$\gets$}

\put(0,5){\line(1,0){15}}
\put(-5,10){\line(0,1){15}}
\put(0,30){\line(1,0){15}}
\put(20,16){\line(0,1){5}}
\qbezier(-5,10)(-5,5)(0,5)
\qbezier(-5,25)(-5,30)(0,30)
\qbezier(15,30)(20,30)(20,25)
\qbezier(15,5)(20,5)(20,10)
\put(20,10){\line(0,1){4}}
\put(20, 23){\line(0,1){2}}

\put(19, 15){\line(1,0){22}}
\put(19, 22){\line(1,0){20}}
\qbezier(17,13)(17,15)(19,15)
\qbezier(17,13)(17,11)(19,11)
\qbezier(17,24)(17,22)(19,22)
\qbezier(17,24)(17,26)(19,26)

\put(21,11){\line(1,0){20.5}}
\put(21,26){\line(1,0){18}}
\put(41,26){\line(1,0){0.5}}
\qbezier(41,15)(44.5,15)(44.5,18.5)
\qbezier(41,22)(44.5,22)(44.5,18.5)
\qbezier(41.5,11)(49,11)(49,18.5)
\qbezier(41.5,26)(49,26)(49,18.5)

\put(45,5){\line(1,0){15}}
\put(40,16){\line(0,1){9}}
\put(45,30){\line(1,0){15}}

\qbezier(40,10)(40,5)(45,5)
\qbezier(40,25)(40,30)(45,30)
\qbezier(60,30)(65,30)(65,25)
\qbezier(60,5)(65,5)(65,10)

\put(61,13.5){\line(1,0){8}}
\put(61,21.5){\line(1,0){8}}
\put(61, 13.5){\line(0,1){8}}
\put(69, 13.5){\line(0,1){8}}
\put(65,10){\line(0,1){3.5}}
\put(65, 21.5){\line(0,1){3.5}}

\put(23,-3){Figure 1.}
\end{picture}

The group $H_1(M)$ is freely generated by the images of the meridians $\{a,p,x\}$.
The 3-manifold $M$ has a JSJ decomposition of which one piece 
is homeomorphic to $D_{oo}\times{S^1}$, 
where $D_{oo}$ is the ``pair of pants", i.e., the twice punctured disc. 
Two of the boundary components of this piece are identified, 
to give a copy of $T_o\times{S^1}$ in $M$.
Homology considerations show that the  JSJ decomposition has no other such piece, 
and so self-homeomorphisms of $M$ must leave this piece invariant,
up to isotopy. 
In particular, there is no self-homeomorphism carrying the image of $x$
in $H_1(M)$ into the subgroup generated by the other meridians.
Since $j_{Y*}(x)$ generates $H_1(Y)$ and $j'_{Y*}(x)=0$,
it follows that $j$ and $j'$ are not equivalent.

If we tie distinct hyperbolic $Ap1$ knots $K_1,K_2$ and $K_3$ in each component of $Bo$ 
then the three partitions $\mathcal{P}_{12}=\{\{K_1,K_2\}, K_3\}$, $\mathcal{P}_{13}=\{\{K_1,K_3\}, K_2\}$.
and $\mathcal{P}_{23}=\{K_1,\{K_2,K_3\}\}$ are distinct.
The JSJ decomposition of $M(L)$ has 4 pieces: $X(K_1)$, $X(K_2)$, $X(K_3)$ and $X(Bo)$.
These are distinct, by Lemma 1 below, 
and so $M$ has three inequivalent abelian embeddings.

\begin{lemma}
Let $K$ be a knot in $S^3$ and let $\mu_K$ be a meridian loop for $K$.
Then the $3$-manifold $N$ with boundary $T$ obtained by $0$-framed surgery on $K$
in the exterior $X(\mu_K)\cong{S^1\times{D^2}}$ is homeomorphic to $X(K)$.
\end{lemma}

\begin{proof}
The cocore of the surgery on $S^3$ giving $M(K)$ is isotopic to the image of $\mu_K $
in $X(K)\subset{M(K)}$. 
Deleting a regular neighbourhood of this core from $M(K)$ gives back $X(K)$.
\end{proof}

For the other cases we shall need links with at least 4 components.

If $\beta=2$ then $\pi_X\cong\pi_Y\cong\Z\oplus{C_n}$, for some $n\geq1$,
and we start by replacing $Bo_3$ by its $(2,2n)$-cable.
If $\beta=4$ then $\pi_X\cong\pi_Y\cong\Z^2$,
and we replace $Bo_3$ by two parallel unlinked components.
Call the resulting link $Bo(+)$.

In each case we then insert nontrivial $Ap1$ knots into the second and fourth components
of $Bo(+)$.
The new 4-component link $L$ has two partitions into split sublinks which are also slice links:
\[
\mathcal{P}=\{\{L_1,L_3\}, \{L_2,L_4\}\}\quad\mathrm{and}\quad\mathcal{P'}=\{\{L_1,L_4\}, \{L_2,L_3\}\}.
\]
The associated embeddings are abelian, and the JSJ argument again goes through
(provided $n\not=1$).

If $n=0$ (so $\beta=4$) then each of the 2-component sublinks
of $Bo(+)$ is trivial, but the embedding associated to the partition
$\mathcal{P}=\{\{L_1,L_2\}, \{L_3,L_4\}\}$ is not abelian.
We may modify the second component of $Bo(+)$, as in Figure 3, 
so that it represents the commutator of the meridians 
of the third and fourth components. 
Each of the 2-component sublinks of the resulting link remains trivial.
Now tie distinct hyperbolic $Ap1$ knots 
in each of the second, third and fourth components.
Then the 3 partitions of the resulting $L$ into pairs of disjoint 2-component sublinks
each give rise to an abelian embedding of $M(L)$,
and once again these embeddings are distinct.

The case $\beta=6$ involves a little more effort.
In \cite[Example 8.3]{LF} we considered the links obtained as preimages of the Whitehead link $Wh$
under 2- and 3-fold branched cyclic covers of $S^3$, branched over an unknotted axis.
The associated manifolds $M(L)$ have abelian embeddings.
However these links do not have  distinct partitions leading to abelian embeddings.
The 3-fold cover of $Wh$ (with respect to this branching) is the link of Figure 2.

\setlength{\unitlength}{1mm}
\begin{picture}(95,80)(-21,-8)

\put(-9,63){$A$}
\put(56,63){$B$}
\put(55,10){$C$}
\put(3,23){$S$}
\put(85,23){$R$}
\put(40,40){$T$}

\put(40,35){$\bullet$}

\linethickness{1pt}
\put(-1,65){\line(1,0){19}}
\put(-4,43){\line(0,1){20}}
\put(-1,40){\line(1,0){9}}
\put(10,40){\line(1,0){4}}
\qbezier(-4,62)(-4,65)(-1,65)
\qbezier(-4,43)(-4,40)(-1,40)
\qbezier(18,65)(21,65)(21,62)
\qbezier(18,40)(21,40)(21,43)
\put(21,43){\line(0,1){7}}
\put(21,52){\line(0,1){4}}
\put(21,57.8){\line(0,1){1.4}}
\put(21,61){\line(0,1){1}}

\put(64,65){\line(1,0){19}}
\put(61,48){\line(0,1){5}}
\put(61,55){\line(0,1){8}}
\put(64,40){\line(1,0){8}}
\put(74,40){\line(1,0){4}}
\qbezier(61,62)(61,65)(64,65)
\qbezier(61,43)(61,40)(64,40)
\qbezier(83,65)(86,65)(86,62)
\qbezier(83,40)(86,40)(86,43)
\put(86,43){\line(0,1){19}}

\put(32,30){\line(1,0){18}}
\put(32,5){\line(1,0){18}}
\qbezier(29,27)(29,30)(32,30)
\qbezier(29,8)(29,5)(32,5)
\qbezier(50,30)(53,30)(53,27)
\qbezier(50,5)(53,5)(53,8)
\put(29,10){\line(0,1){4}}
\put(29,16){\line(0,1){4}}
\put(29,22){\line(0,1){5}}
\put(53,8){\line(0,1){7}}
\put(53,17){\line(0,1){5}}

\thinlines
\put(9,15){\line(0,1){26}}
\put(12,25){\line(0,1){14}}
\qbezier(9,41)(9,42.5)(10.5,42.5)
\qbezier(10.5,42.5)(12,42.5)(12,41)
\qbezier(9,15)(9,9)(15,9)
\put(15,9){\line(1,0){11}}
\qbezier(12,25)(12,21)(16,21)
\put(16,21){\line(1,0){10}}

\put(20.5,51){\line(1,0){39.5}}
\put(22,47){\line(1,0){54}}
\qbezier(76,47)(79,47)(79,44)

\qbezier(20.5,47)(19,47)(19,48.5)
\put(19,48.5){\line(0,1){1}}
\qbezier(20.5,51)(19,51)(19,49.5)

\put(19,54){\line(1,0){1}}
\qbezier(22,54)(23.5,54)(23.5,55.5)
\qbezier(22,57)(23.5,57)(23.5,55.5)
\put(18,57){\line(1,0){4}}
\qbezier(18,57)(15,57)(15,54)
\put(15,31){\line(0,1){23}}
\qbezier(15,31)(15,24)(22,24)
\put(22,24){\line(1,0){6}}

\qbezier(19,54)(17.5,54)(17.5,52.5)
\put(17.5,32){\line(0,1){20.5}}
\qbezier(17.5,32)(17.5,27)(22.5,27)
\put(22.5,27){\line(1,0){5.5}}
\qbezier(30,24)(31.5,24)(31.5,22.5)
\qbezier(31.5,22.5)(31.5,21)(30,21)
\put(26,21){\line(1,0){4}}

\put(14,60){\line(1,0){16}}
\qbezier(30,60)(33,60)(33,57)
\qbezier(33,57)(33,54)(36,54)
\put(36,54){\line(1,0){26}}
\qbezier(62,54)(63.5,54)(63.5,52.5)
\qbezier(62,51)(63.5,51)(63.5,52.5)

\qbezier(6,52)(6,60)(14,60)

\put(18.5,44){\line(1,0){1.5}}
\put(22,44){\line(1,0){52}}
\qbezier(74,44)(76,44)(76,42)
\put(82,22){\line(0,1){26}}

\put(79,39){\line(0,1){5}}

\put(73,38){\line(0,1){3}}
\qbezier(72,42)(73,42)(73,41)
\qbezier(70,41)(70,42)(71,42)
\put(71,42){\line(1,0){1}}

\put(66,41){\line(0,1){2}}
\qbezier(66,48)(66,55)(73,55)
\put(73,55){\line(1,0){2}}
\qbezier(75,55)(82,55)(82,48)

\qbezier(6,52)(6,44)(14,44)
\qbezier(79,39)(79,37.5)(77.5,37.5)
\qbezier(76,39)(76,37.5)(77.5,37.5)

\put(66,32){\line(0,1){7}}
\qbezier(60,26)(66,26)(66,32)
\put(70,30){\line(0,1){9}}
\qbezier(63,23)(70,23)(70,30)
\put(50,23){\line(1,0){15}}
\put(47,26){\line(1,0){13}}

\put(73,23){\line(0,1){17}}
\put(54,19){\line(1,0){15}}
\qbezier(69,19)(73,19)(73,23)
\put(52,16){\line(1,0){24}}
\qbezier(76,16)(82,16)(82,22)
\qbezier(52,16)(50.5,16)(50.5,17.5)
\qbezier(50.5,17.5)(50.5,19)(52,19)

\put(26,9){\line(1,0){10}}
\qbezier(28,15)(26.5,15)(26.5,16.5)
\qbezier(26.5,16.5)(26.5,18)(28,18)
\put(30,27){\line(1,0){6}}
\qbezier(36,27)(41,27)(41,22)
\qbezier(36,9)(41,9)(41,14)
\put(41,19){\line(0,1){3}}

\put(28,15){\line(1,0){16}}
\put(30,18){\line(1,0){11}}
\qbezier(41,18)(44,18)(44,21)
\put(44,21){\line(0,1){2}}
\qbezier(44,23)(44,26)(47,26)
\qbezier(44,15)(47,15)(47,18)
\put(47,18){\line(0,1){2}}
\qbezier(47,20)(47,23)(50,23)

\put(20,-5){Figure 2.   A 6-component link}

\end{picture}

We shall label the components of the preimage of one component 
of this link with $A,B,C$ and the other components  with $R,S$ and $T$.
Each of the six consecutive triples $\{A,T,B\}$, $\{T,B,R\}$, $\{B,R,C\}$, $\{R,C,S\}$, $\{C,S,A\}$
and $\{S,A,T\}$ is a nontrivial Brunnian link,
while all 2-component sublinks and all the other 3-component sublinks are trivial.
Each component represents the commutator of the meridians of its immediate neighbours
(up to inversion).
Thus the embedding determined by $\mathcal{P}=\{\{A,B,C\},\{R,S,T\}\}$ 
and the obvious set of slice discs is abelian.

We may modify $B, C,R$ and $S$ to represent $[c,s],[a,r],[c,t]$ and $[b,r]$
(up to inversion),
while the only nontrivial 3-component sublinks are
$\{A,T,B\}$, $\{T,B,R\}$, $\{B,R,C\}$, $\{R,C,S\}$, $\{C,S,A\}$
and $\{S,A,T\}$, and $\{A,C,R\},\{B,C,S\}, \{B,R,S\}$ and $\{C,R,S\}$.
Figure 3 shows only the modification to $R$. 
(Note that the trivial link $\{C,R,T\}$ becomes a copy of $Bo$,
if we ignore the other 3 components.)

\setlength{\unitlength}{1mm}
\begin{picture}(95,82)(-21,-10)

\put(-9,63){$A$}
\put(56,63){$B$}
\put(55,10){$C$}
\put(5,27){$S$}
\put(34,40){$T$}

\linethickness{1.3pt}
\put(-1,65){\line(1,0){19}}
\put(-4,43){\line(0,1){20}}
\put(-1,40){\line(1,0){9}}
\put(10,40){\line(1,0){4}}
\qbezier(-4,62)(-4,65)(-1,65)
\qbezier(-4,43)(-4,40)(-1,40)
\qbezier(18,65)(21,65)(21,62)
\qbezier(18,40)(21,40)(21,43)
\put(21,43){\line(0,1){7}}
\put(21,52){\line(0,1){4}}
\put(21,57.8){\line(0,1){1.4}}
\put(21,61){\line(0,1){1}}

\put(64,65){\line(1,0){19}}
\put(61,45){\line(0,1){2}}
\put(61,49){\line(0,1){4}}
\put(61,55){\line(0,1){8}}
\put(64,40){\line(1,0){8}}
\put(73.6,40){\line(1,0){4.8}}
\put(79.8,40){\line(1,0){1.4}}
\qbezier(61,62)(61,65)(64,65)
\qbezier(61,43)(61,40)(64,40)
\qbezier(83,65)(86,65)(86,62)
\qbezier(83,40)(86,40)(86,43)
\put(86,43){\line(0,1){19}}

\put(32,30){\line(1,0){10}}
\put(43.7,30){\line(1,0){1.6}}
\put(46.7,30){\line(1,0){3.3}}
\put(32,5){\line(1,0){18}}
\qbezier(29,27)(29,30)(32,30)
\qbezier(29,8)(29,5)(32,5)
\qbezier(50,30)(53,30)(53,27)
\qbezier(50,5)(53,5)(53,8)
\put(29,10){\line(0,1){4}}
\put(29,16){\line(0,1){4}}
\put(29,22){\line(0,1){5}}
\put(53,8){\line(0,1){7}}
\put(53,17){\line(0,1){5}}

\thinlines
\put(9,14){\line(0,1){27}}
\put(12,27){\line(0,1){12}}
\qbezier(9,41)(9,42.5)(10.5,42.5)
\qbezier(10.5,42.5)(12,42.5)(12,41)
\qbezier(9,14)(9,9)(14,9)
\put(14,9){\line(1,0){12}}
\qbezier(12,27)(12,21)(18,21)
\put(18,21){\line(1,0){8}}

\put(20.5,51){\line(1,0){39.5}}
\put(22,48){\line(1,0){54}}
\qbezier(76,48)(79,48)(79,45)

\qbezier(20.5,48)(19,48)(19,49.5)
\qbezier(20.5,51)(19,51)(19,49.5)

\put(19,54){\line(1,0){1}}
\qbezier(22,54)(23.5,54)(23.5,55.5)
\qbezier(22,57)(23.5,57)(23.5,55.5)
\put(18,57){\line(1,0){4}}
\qbezier(18,57)(15,57)(15,54)
\put(15,31){\line(0,1){23}}
\qbezier(15,31)(15,24)(22,24)
\put(22,24){\line(1,0){6}}

\qbezier(19,54)(17.5,54)(17.5,52.5)
\put(17.5,32){\line(0,1){20.5}}
\qbezier(17.5,32)(17.5,27)(22.5,27)
\put(22.5,27){\line(1,0){5.5}}
\qbezier(30,24)(31.5,24)(31.5,22.5)
\qbezier(31.5,22.5)(31.5,21)(30,21)
\put(26,21){\line(1,0){4}}

\put(14,60){\line(1,0){31}}
\qbezier(45,60)(48,60)(48,57)
\qbezier(48,57)(48,54)(51,54)
\put(51,54){\line(1,0){11}}
\qbezier(62,54)(63.5,54)(63.5,52.5)
\qbezier(62,51)(63.5,51)(63.5,52.5)

\qbezier(6,52)(6,60)(14,60)

\put(18.5,44){\line(1,0){1.5}}
\put(22,44){\line(1,0){18}}
\put(44,44){\line(1,0){30}}
\qbezier(74,44)(76,44)(76,42)

\qbezier(6,52)(6,44)(14,44)
\qbezier(79,39)(79,37.5)(77.5,37.5)
\qbezier(76,39)(76,37.5)(77.5,37.5)
\put(79,39){\line(0,1){6}}

\put(26,9){\line(1,0){10}}
\put(30,27){\line(1,0){6}}
\qbezier(36,27)(41,27)(41,22)
\qbezier(36,9)(41,9)(41,14)
\put(41,19){\line(0,1){3}}

\linethickness{0.8pt}
\put(83,17){$R$}

\qbezier(28,15)(26.5,15)(26.5,16.5)
\qbezier(26.5,16.5)(26.5,18)(28,18)
\put(28,15){\line(1,0){16}}
\put(30,18){\line(1,0){11}}
\qbezier(41,18)(44,18)(44,21)
\put(44,21){\line(0,1){2}}
\qbezier(44,23)(44,26)(47,26)
\put(47,26){\line(1,0){13}}
\qbezier(60,26)(66,26)(66,32)
\put(66,32){\line(0,1){1}}
\put(66,37){\line(0,1){2}}

\qbezier(44,15)(47,15)(47,18)
\put(47,18){\line(0,1){2}}
\qbezier(47,20)(47,23)(50,23)
\put(50,23){\line(1,0){15}}
\put(70,30){\line(0,1){9}}
\qbezier(63,23)(70,23)(70,30)

\put(73,23){\line(0,1){17}}
\put(54,19){\line(1,0){15}}
\qbezier(69,19)(73,19)(73,23)
\put(52,16){\line(1,0){24}}
\qbezier(76,16)(82,16)(82,22)
\qbezier(52,16)(50.5,16)(50.5,17.5)
\qbezier(50.5,17.5)(50.5,19)(52,19)

\put(66,41){\line(0,1){2}}
\qbezier(66,49)(66,56)(73,56)
\put(73,56){\line(1,0){2}}
\qbezier(75,56)(82,56)(82,49)
\put(82,22){\line(0,1){27}}

\put(73,38){\line(0,1){3}}
\qbezier(72,42)(73,42)(73,41)
\qbezier(70,41)(70,42)(71,42)
\put(71,42){\line(1,0){1}}

\put(41,31){\line(0,1){13}}
\put(43,29){\line(0,1){15.5}}
\qbezier(41,29)(41,28)(42,28)
\qbezier(42,28)(43,28)(43,29)
\qbezier(41,44)(41,47)(44,47)
\put(44,47){\line(1,0){1}}
\qbezier(45,47)(48,47)(48,45)

\qbezier(43,44)(43,45.5)(44.5,45.5)
\put(44,45.5){\line(1,0){1}}
\qbezier(45,45.5)(46,45.5)(46,44.5)
\qbezier(46,29)(46,28)(47,28)
\qbezier(47,28)(48,28)(48,29)

\put(46,29){\line(0,1){14}}
\put(48,31){\line(0,1){2}}
\qbezier(48,33)(48,34)(49,34)
\qbezier(48,37)(48,36)(49,36)
\put(48,37){\line(0,1){6}}
\put(49,34){\line(1,0){16}}
\put(49,36){\line(1,0){16}}
\qbezier(65,34)(66,34)(66,33)
\qbezier(65,36)(66,36)(66,37)

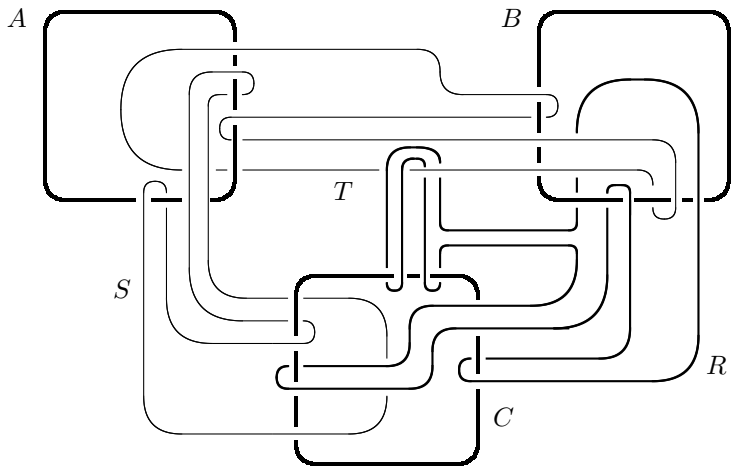
\put(0,-5){Figure 3.   Modifying $R$ so that it represents $[c,t]$.}

\end{picture}

After further modifications to $B,C$ and $S$ we  insert distinct nontrivial hyperbolic
$Ap1$ knots into each of the components $R,S$ and $T$, to obtain 
a link with two partitions $\mathcal{P}$ and $\mathcal{P'}=\{\{A,B,R\},\{C,S,T\}\}$ 
which each give rise to abelian embeddings $j$ and $j'$ of $M(L)$.
The JSJ decomposition of $M(L)$ has 3 distinct hyperbolic components corresponding 
to $R,S$ and $T$, by Lemma 1.
It follows as before that $M(L)$ has no self-homeomorphisms which  permute the basis
of $H_1(M)$ in a manner necessary for an equivalence between $j$ and $j'$.

A 6-component link $L=\{A,B,C,R,S,T\}$ has 10 partitions into pairs of 3-component sublinks.
If each component of one sublink of a partition represents a commutator of meridians 
of two components of the other sublink then some of these partitions cannot represent abelian embeddings.
Suppose for example that $\mathcal{P}=\{\{A,B,C\},\{R,S,T\}\}$ is a partition such that 
$A$ represents the commutator $[s,t]$ of the meridians for $S$ and $T$ in the exterior of $\{R,S,T\}$. 
Then $\{A,S,T\}$ cannot be a slice link,
since the nilpotent completion of a slice link group is that of a free group -- see \cite[Chapter 12.7]{AIL}.
Consideration of the combinatorics of the problem 
then suggests that at most 5 of the partitions could give rise to abelian embeddings.
We may start with the following partitions of a trivial 6-component link into pairs of trivial 3-component links:

$\mathcal{P}_1=\{\{A,B,C\},\{R,S,T\}\}$,
$\mathcal{P}_2=\{\{A,B,R\},\{C,S,T\}\}$,

$\mathcal{P}_3=\{\{A,B,S\},\{C,S,T\}\}$,
$\mathcal{P}_4=\{\{A,C,R\},\{B,S,T\}\}$ and

$\mathcal{P}_5=\{\{A,C,T\},\{B,R,S\}\}$.\\
We then modify each of the  other ten 3-component sublinks as in Figure 3 to
obtain copies of $Bo$, and tie distinct hyperbolic $Ap1$ knots in each of 5 of the components.
Once again we may use the uniqueness of the JSJ decomposition to show
that the 5 abelian embeddings corresponding to these partitions are inequivalent.
We suspect that dropping the hypothesis on components representing commutators
would not lead to more than 5 distinct embeddings, but have no proof for this.

Are there similar examples when $H_1(M)\cong{C_n^2}$ or $\Z^2$? 
Here arguments involving automorphisms of $H_1(M)$ do not seem to be adequate.
 
\section{some remarks on identifying the complementary regions when $\beta=0$}

In \cite[\S8.6]{LF} we made some brief observations about the application of surgery 
to examples of abelian embeddings of 3-manifolds $M$ with $H_1(M)\cong{C_n^2}$ in $S^4$, 
for the cases $n\leq4$.
We shall remove the latter restriction here.
Let $W$ be a complementary region of an embedding of a  rational homology sphere $M$ in $S^4$,
such that $\pi_W=\pi_1W\cong{C_n}$, for some $n\geq1$.
Then $W\simeq{P_n}=S^1\cup_ne^2$ \cite[Lemma 8.4]{LF}, 
and homotopy equivalences between such pseudoprojective planes are simple \cite{DS73}.
Let $\mathcal{S}_{TOP}(W,M)$ be the simple-homotopy equivalence structure set.
The normal invariants are detected by the signature, 
since $H^2(W,M;\mathbb{F}_2)=H_2(W;\mathbb{F}_2)=0$.
Hence exactness of the surgery sequence implies that $L_1^s(\Z[\pi_W])$
acts transitively on $\mathcal{S}_{TOP}(W,M)$.
If moreover $\pi_W$ has odd order then $L^s_1(\Z[\pi_W])=0$ \cite{Ba75},
while if $n$ is even $L^s_1(\Z[\pi_W])$ is a finite 2-group.
Thus if $n$ is odd the complementary regions of an abelian embedding of $M$ 
are determined up to homeomorphism by $M$ and the homotopy types 
of the inclusions of $M$ as their boundaries.

The main difficulty in determining the abelian embeddings of such 3-manifolds is
in computing the group $[M,P_n]_f$ of based homotopy classes of based maps inducing a given epimorphism $f:\pi_1M\to{C_n}=\pi_1P_n$.

Work of \cite{La84} implies that $S^3/Q(8)$  has an essentially unique abelian embedding. 


\end{document}